\documentclass{article}
\usepackage{amssymb, amsthm, amsmath, amsfonts}
\usepackage{geometry}
\usepackage[colorlinks=true,urlcolor=blue,linkcolor=red,citecolor=magenta]{hyperref}
\usepackage{cleveref}
\usepackage{hyperref}
\usepackage{framed}

\usepackage{xcolor}

\newtheorem{theorem}{Theorem}[section]
\newtheorem{lemma}[theorem]{Lemma}
\newtheorem{corollary}[theorem]{Corollary}
\newtheorem{proposition}[theorem]{Proposition}

\theoremstyle{definition}

\DeclareMathOperator{\wt}{wt} 

\DeclareMathOperator{\mult}{mult}

\newcommand{\graph}[1]{\mathcal{G}_{#1}}

\newcommand{\F}{\mathbb{F}}
\newcommand{\set}[1]{\left\{ #1 \right\}}

\newcommand{\parens}[1]{ \left( #1 \right)}

\newcommand{\floor}[1]{ \left\lfloor #1 \right\rfloor }

\newcommand{\maxmults}[1]{\mathcal{M}_{#1}}

\begin{document}
\title{A new upper bound for Sidon sets in $\mathbb{F}_2^{4k+3}$}
\author{Darrion Thornburgh}
\date{Department of Mathematics, Vanderbilt University, USA.
\texttt{darrion.thornburgh@vanderbilt.edu}\\[2ex]%
}
\maketitle

\begingroup
\renewcommand{\thefootnote}{}
\addtocounter{footnote}{-1}
\endgroup

\begin{abstract}
A subset $S \subseteq \mathbb{F}_2^n$ is called a Sidon set if no four distinct points of $S$ have zero sum. It is shown that if $n \geq 7$ and $n \equiv 3 \mod 4$, then $|S|\leq 2^{\frac{n+1}{2}}-3$. As a consequence, for any even $t \geq 4$, there does not exist a binary linear $[2^t-3,2^t-2t-2,5]$-code, strengthening a nonexistence result of Brouwer and Tolhuizen from 1993. As a second consequence, we also show that for an even integer $n \geq 4$, every almost perfect nonlinear (APN) function $F \colon \mathbb{F}_2^n \to \mathbb{F}_2^n$ has nonlinearity at least $3$.
\end{abstract}

\section{Introduction}
A (binary) \textit{Sidon set} is a subset $S$ of $\F_2^n$ such that the equation $x+y=z+w$ is never satisfied for any distinct $x,y,z,w \in S$.
Equivalently, Sidon sets in $\F_2^n$ are those subsets containing no 2-dimensional affine subspaces, making the problem of determining the largest size of a Sidon set in $\F_2^n$ analogous to the famous \textit{cap set problem}, which is to determine the largest size of a subset of $\F_3^n$ containing no affine line \cite{croot2024pastfuturecapset}.
For the best-known constructions of large Sidon sets in $\F_2^n$, see \cite{CzerwinskiPottLargeSidon}.

Denote by $s_{\max}(\F_2^n)$ the largest size of a Sidon set in $\F_2^n$.
One can obtain the following trivial upper bound on $s_{\max}(\F_2^n)$ via a simple counting argument (see, for instance, \cite{CzerwinskiPott2024}):
\[
s_{\max}(\F_2^n) \leq 
\begin{cases}
    2^{\frac{n+1}{2}} & \text{if $n$ is odd},\\
    \floor{2^{\frac{n+1}{2}}+\frac{1}{2}} & \text{if $n$ is even}.
\end{cases}
\]
For the case of $n$ even, the first and only (at the time this paper was written) improvement on the trivial upper bound was given in \cite{CzerwinskiPott2024} by Czerwinski and Pott, who showed that if some technical conditions are satisfied, one can improve upon the trivial bound by $1$ or $2$.
For $n\geq 7$ odd, the trivial bound can be improved by a nonexistence result on binary linear codes from 1993 by Brouwer and Tolhuizen \cite{Brouwer1993}, providing $s_{\max}(\F_2^n)\leq 2^{\frac{n+1}{2}}-2$.
However, the upper bound on the size of Sidon sets in odd dimensions has not been improved since the result of Brouwer and Tolhuizen from 1993.
The main result of this paper is that for all $n \geq 7$ satisfying $n \equiv 3 \mod 4$, we improve the best known upper bound on the size of a Sidon set in $\F_2^n$ by exactly one.

\begin{theorem}\label{thm:main}
    Assume $n \geq 7$ and $n \equiv 3 \mod 4$.
    Then $s_{\max}(\F_2^n) \leq 2^{\frac{n+1}{2}}-3$.
\end{theorem}

It is worth mentioning that, although \Cref{thm:main} is the first improvement in more than three decades to the upper bound on Sidon sets in $\F_2^{4k+3}$, the gap between the best-known lower and upper bounds still remains large (this is also true in the case of even dimensions).
In particular, in odd dimensions, we lack constructions of infinite families of large Sidon sets, and the best recent improvements have come from considering intersections of graphs of almost perfect nonlinear (APN) functions with affine hyperplanes \cite{CzerwinskiPottLargeSidon}.

We prove \Cref{thm:main} by showing that for $n\geq6$, the value of $s_{\max}(\F_2^n)$ is strictly less than the smallest integer $r_n$ such that $r_n \equiv 2 \mod 6$ and $r_n > \sqrt{2^{n+1}-8}-2$.
Unlike \cite{Brouwer1993}, we do not use a coding-theoretic approach, but instead, we introduce a new structural constraint on Sidon sets whose sizes are congruent to $2$ modulo $6$.
In particular, for a Sidon set $S \subseteq \F_2^n$ such that $|S| \equiv 2 \mod 6$, we derive lower and upper bounds (the latter via a structural argument) on the number of points $a \in \F_2^n \setminus S$ such that there are exactly $\frac{|S|-2}{3}$ distinct subsets of size $3$ in $S$ with sum equal to $a$ (which is the maximal amount when $|S| \equiv 2 \mod 6$).
We then derive a contradiction by showing that these lower and upper bounds cannot hold if the size of $S$ is too large.

We discuss in \Cref{sec:cor} two corollaries of \Cref{thm:main}.
The first states that, for all even $t \geq 4$, there does not exist a binary linear $[2^t-3,2^t-2t-2,5]$-code, improving the nonexistence result of Brouwer and Tolhuizen.
For our second corollary, we prove the nonlinearity of any APN function in an even number $n \geq 4$ of variables is at least $3$ (which was previously known to be at least $2$, see \cite{CzerwinskiPottLargeSidon}).

\section{Notation}\label{sec:notation}
For any subset $S \subseteq \F_2^n$, let $S+S = \set{s_1 + s_2 : s_1, s_2 \in S}$ and $S\widehat{+}S = (S+S) \setminus \set{0}$.
By definition, $S$ is a Sidon set if and only if $|S\widehat{+}S| = \binom{|S|}{2}$ since the sums of two distinct pairs of points in $S$ are not equal if and only if $S$ is Sidon.
For any point $a \in \F_2^n$, we let $\Delta_S(a) = |S \cap (a+S)|$.
Also, for any positive integer $k$ such that $1 \leq k \leq |S|$, let $\binom{S}{k}$ be the family consisting of all subsets of $S$ of size $k$.
For any point $a \in \F_2^n$, let 
\[
\mathcal{T}_{S,k}(a) = \set{\set{x_1, \dots, x_k} \in \binom{S}{k} : \sum_{i=1}^k x_i=a},
\]
and let $\mult_{S,k}(a)$ be the size of $\mathcal{T}_{S,k}(a)$. 

When $S \subseteq \F_2^n$ is a Sidon set, any distinct $T_1, T_2 \in \mathcal{T}_{S,3}(a)$ are disjoint if $a \notin S$.
Indeed, if $x_1+x_2 +x_3 = a = x_1+x_4+x_5$ for some $x_1, \dots, x_5 \in S$, then $x_2+x_3+x_4+x_5=0$, implying $\set{x_2, x_3} = \set{x_4,x_5}$ since $S$ is a Sidon set.
In particular, this provides the inequality 
$
\mult_{S,3}(a) \leq \floor{\frac{|S|}{3}}
$
for any point $a \notin S$ (this was also shown in \cite[Corollary 3.4]{quadspaper}).
We denote by $\maxmults{S}$ the set of points $a$ in the complement of $S$ such that $\mult_{S,3}(a)$ reaches this optimal value, i.e.  
\[
\maxmults{S} = \set{a \in \F_2^n \setminus S : \mult_{S,3}(a) = \floor{\frac{|S|}{3}}}.
\]

\section{Proof of \Cref{thm:main}}

As mentioned in the introduction, we are interested in bounding from below and above the size of $\maxmults{S}$ when $S \subseteq \F_2^n$ is a Sidon set such that $|S| \equiv 2 \mod 6$.
Let us first consider $S$ to be an arbitrary Sidon set (with no restrictions on its size).
For any point $a \in \F_2^n \setminus S$, we have by definition $\floor{\frac{|S|}{3}} - \mult_{S,3}(a) \geq 0$ with equality if and only if $a \in \maxmults{S}$.
In particular, this implies that 
$
2^n-|S| - |\maxmults{S}| \leq \sum_{a \in \F_2^n \setminus S} \parens{\floor{\frac{|S|}{3}} - \mult_{S,3}(a)}.
$
Recall from \Cref{sec:notation} that any two elements of $\mathcal{T}_{S,3}(a)$ are disjoint.
Hence, $\sum_{a \in \F_2^n \setminus S} \mult_{S,3}(a) = \binom{|S|}{3}$, and we then have $2^n -|S|-|\maxmults{S}| \leq (2^n-|S|)\floor{\frac{|S|}{3}} -\binom{|S|}{3}$.
Upon rearrangement, we have
\begin{equation}\label{eq:lowerbound-on-maxmults}
    2^n-|S| - (2^n-|S|)\floor{\frac{|S|}{3}}  + \binom{|S|}{3} \leq |\maxmults{S}|.
\end{equation}
In particular, if $|S| \equiv 2 \mod 3$, then $\floor{\frac{|S|}{3}}=\frac{|S|-2}{3}$ and inequality (\ref{eq:lowerbound-on-maxmults}) is equivalent to 
\begin{equation}\label{eq:lowerbound-on-maxmults-size2mod3}
    \frac{|S|^3 -|S|^2 -8|S| -2^{n+1}|S| + 5 \cdot 2^{n+1}}{6} \leq |\maxmults{S}|.
\end{equation}

The remainder of this section is then motivated by a particular structural property of a Sidon set $S$ such that $\maxmults{S}\neq\emptyset$ and $|S| \equiv 2 \mod 6$, and we will use this structural property to derive an upper bound on the size of $\maxmults{S}$.

\begin{proposition}\label{prop:2-nontriple-pts}
    Let $S\subseteq \F_2^n$ be a Sidon set such that $|S| \equiv 2 \mod 6$, and assume that $\maxmults{S} \neq \emptyset$.
    There exists a unique pair of distinct points $s_1, s_2 \in S$ such that for all $a \in \maxmults{S}$,
    \[
    S \setminus \set{s_1, s_2} = \bigcup_{T \in \mathcal{T}_{S,3}(a)}T.
    \]
\end{proposition}
\begin{proof}
Let $a \in \maxmults{S}$.
Then $\mult_{S,3}(a) = \floor{\frac{|S|}{3}}= \ell$, where $\ell = \frac{|S|-2}{3}$.
The union $\bigcup_{T \in \mathcal{T}_{S,3}(a)}T$ has size $|S|-2$, and so let $s_{a,1},s_{a,2}$ be the two points of $S$ not contained in this union. 
Hence, $\mathcal{T}_{S,3}(a)$ is a partition of $S \setminus \set{s_{a,1},s_{a,2}}$ into $\ell$ distinct triples whose sums are equal to $a$, implying
$
\sum_{s \in S} s = s_{a,1}+s_{a,2} + \ell a = s_{a,1}+s_{a,2}
$
because $\ell = \frac{|S|-2}{3}$ is even.
In particular, since $\sum_{s \in S} s$ does not depend on $a$, the value of $s_{a,1}+s_{a,2}$ does not depend on $a$.
The result immediately follows by setting $\set{s_1, s_2} = \set{s_{a,1},s_{a,2}}$ with uniqueness holding as $S$ is Sidon. 
\end{proof}

We now begin deriving our upper bound on the size of $\maxmults{S}$ using the previous result.

\begin{lemma}\label{lem:maxmults-upperbound-by-deltaA}
    Let $S\subseteq \F_2^n$ be a Sidon set such that $|S| \equiv 2 \mod 6$, and assume that $\maxmults{S} \neq \emptyset$.
    Let $s_1, s_2 \in S$ be defined as in \Cref{prop:2-nontriple-pts}.
    Then $|\maxmults{S}| \leq \Delta_A(s_1+s_2)$ where $A = \F_2^n \setminus (S+S)$.
\end{lemma}
\begin{proof}
If $a \notin S$ and $s \in S$, then $s \in \bigcup_{T \in \mathcal{T}_{S,3}(a)}T$ if and only if $a+s \notin A$ because the latter condition is the same as $a \in s+(S+S)$.
By \Cref{prop:2-nontriple-pts}, we know $S \setminus \set{s_1,s_2} = \bigcup_{T \in \mathcal{T}_{S,3}(a)} T$ for any $a \in \maxmults{S}$.
Thus, for any $a \in \maxmults{S}$, we have $a+s_1\in A$ and $a+s_2 \in A$.
We then have
$
|\maxmults{S}| \leq |(A +s_1) \cap (A+s_2)| = |A \cap (A +s_1+s_2)|=\Delta_A(s_1+s_2)
$.
\end{proof}

Under the same assumptions and notation as \Cref{lem:maxmults-upperbound-by-deltaA}, we see that to determine an upper bound on $|\maxmults{S}|$, it suffices to determine an upper bound on $\Delta_A(s_1+s_2)$.
However, this is equivalent to determining an upper bound on $\Delta_D(s_1+s_2)$ where $D = S \widehat{+}S$.

\begin{lemma}\label{lem:deltaA-deltaD}
    Let $S \subseteq \F_2^n$, let $D = S \widehat{+}S$, and let $A =\F_2^n \setminus (S+S)$.
    For any $d \in D$, we have $\Delta_A(d) = 2^n - 2|D| + \Delta_D(d)$.
\end{lemma}
\begin{proof}
Let $d \in D$.
By definition, $|A| - \Delta_A(d) = |A \cap (\F_2^n \setminus (d+A))|$.
    Since $\F_2^n \setminus A = S+S = D \cup \set{0}$ and $a+d\neq 0$ for any $a \in A$, it follows that $|A \cap (\F_2^n \setminus (d+A))| = |A \cap (d+D)|$.
    From these first two equalities, we have $|A| -\Delta_A(d) =  |A \cap (d+D)|$.
    Note that $A,D, \set{0}$ partition $\F_2^n$, so $d+D$ is equal to the disjoint union 
    \[
    (A \cap (d+D)) \cup (D \cap (d+D)) \cup (\set{0} \cap (d+D)),
    \]
    which then has size $|D|= |A|-\Delta_A(d) + \Delta_D(d) + 1$. 
    The result follows from the equality  $|A|=2^n-|D|-1$.
\end{proof}

In the particular case that $S \subseteq \F_2^n$ is a Sidon set and $s_1, s_2 \in S$ are distinct, the value of $\Delta_D(s_1+s_2)$ can be computed from $\mult_{S,4}(s_1+s_2)$ (recall that $\mult_{S,4}(s_1+s_2)$ is the number of subsets of size $4$ of $S$ that have sum equal to $s_1+s_2$) in the following way.

\begin{lemma}\label{lem:deltaD-multS4}
    Let $S \subseteq \F_2^n$ be a Sidon set, let $s_1, s_2 \in S$ be distinct, and let $d = s_1+s_2$.
    Let $D = S\widehat{+}S$.
    Then $\Delta_D(d)=2(|S|-2)+6\mult_{S,4}(d)$.
\end{lemma}
\begin{proof}
    Consider an arbitrary point $a \in D \cap (d+D)$.
    Then, there exist $X,Y \in \binom{S}{2}$ such that $a =\sum_{x \in X} x$ and $a+d = \sum_{y \in Y}y$, and moreover, we know that both $X$ and $Y$ are uniquely defined since $S$ is Sidon.
    Note that $d = \sum_{x \in X} x + \sum_{y \in Y} y$.
    So, $\Delta_D(d)$ is equal to the number of ordered pairs $(X,Y) \in \binom{S}{2}\times \binom{S}{2}$ such that $d = \sum_{x \in X} x + \sum_{y \in Y} y$.
     
    Let us count the number $N_1$ of $(X,Y) \in \binom{S}{2}\times \binom{S}{2}$ such that $d = \sum_{x \in X} x + \sum_{y \in Y} y$ and $X \cap Y \neq \emptyset$ (note that this implies $X$ and $Y$ have an intersection size of exactly $1$ since $d \neq 0$).
    Assume that $X= \set{x_1, x_2} \in \binom{S}{2}$ and $Y = \set{y_1, y_2} \in \binom{S}{2}$ such that $d=x_1+x_2+y_1+y_2$, and without loss of generality assume $x_1 = y_1$.
    Then, $x_2+y_2=s_1+s_2$, implying $\set{x_2, y_2} = \set{s_1,s_2}$ as $S$ is Sidon.
    So, $\set{\set{x_1, x_2}, \set{y_1, y_2}} = \set{\set{x,s_1}, \set{x,s_2}}$ for some $x \in S \setminus \set{s_1, s_2}$.
    Since there are exactly $|S|-2$ possible choices for $x$, it follows that $N_1=2(|S|-2)$.
    
    The number $N_2$ of $(X,Y) \in \binom{S}{2}\times \binom{S}{2}$ such that $d = \sum_{x \in X} x + \sum_{y \in Y} y$ and $X \cap Y = \emptyset$ is clearly equal to $6\mult_{S,4}(d)$.
    Indeed, any $Z \in \binom{S}{4}$ satisfying $d=\sum_{z \in Z} z$ corresponds to exactly $\binom{4}{2}=6$ ordered pairs $(X,Y) \in \binom{S}{2}\times \binom{S}{2}$ such that $d = \sum_{x \in X} x + \sum_{y \in Y} y$.
    Hence, $\Delta_D(d) = N_1+N_2 = 2(|S|-2)+6\mult_{S,4}(d)$.
\end{proof}

We now provide an upper bound on $\mult_{S,4}(d)$ when $S \subseteq\F_2^n$ is a Sidon set such that $\maxmults{S} \neq \emptyset$ and $|S| \equiv 2 \mod 6$ and $d=s_1+s_2$, where $s_1, s_2 \in S$ are defined as in \Cref{prop:2-nontriple-pts}.
Indeed, we will use this upper bound to then derive our upper bound on $|\maxmults{S}|$.

\begin{lemma}\label{lem:multS4-bound}
    Let $S\subseteq \F_2^n$ be a Sidon set such that $|S| \equiv 2 \mod 6$, and assume that $\maxmults{S} \neq \emptyset$.
    Let $s_1, s_2 \in S$ be defined as in \Cref{prop:2-nontriple-pts}, and let $d = s_1+s_2$.
    Then 
    \[
    \mult_{S,4}(d) \leq \frac{(|S|-2)(|S|-8)}{12}.
    \]
\end{lemma}
\begin{proof}
    Let $S' =S \setminus \set{s_1, s_2}$.
    First, let us show that $\mult_{S,4}(d) = \mult_{S',4}(d)$.
    It suffices to show that there cannot be four distinct points in $S$, containing at least one of $s_1$ or $s_2$, whose sum is $d=s_1+s_2$.
    By way of contradiction, assume $\set{x_1, \dots, x_4} \in \binom{S}{4}$ such that $x_1 + \cdots +x_4 = s_1+s_2$ and 
    $
    \set{x_1, \dots, x_4} \cap \set{s_1, s_2} \neq \emptyset.
    $
    If $|\set{x_1, \dots, x_4} \cap \set{s_1, s_2}| = 2$, then, without loss of generality, we have $x_3 + x_4 = 0$, a contradiction. 
    Otherwise, we have, without loss of generality, that $x_1=s_1$ implying $x_2+x_3+x_4=s_2$, contradicting the assumption that $S$ is Sidon. 
    Hence, $\mult_{S,4}(d) = \mult_{S',4}(d)$.
    
    For $t \in S'$, let $k_t^{(4)} = |\set{T \in \mathcal{T}_{S,4}(d) : t \in T}|$.
    Clearly, we have $4\mult_{S',4}(d)= \sum_{t \in S'} k_t^{(4)}$.
    For any $t \in S'$, if $T \in \mathcal{T}_{S,4}(d)$ such that $t \in T$, then the triple $T \setminus \set{t}$ is in $\mathcal{T}_{S,3}(t+d)$.
    On the other hand, if $T \in \mathcal{T}_{S,3}(t+d)$, then $T \cup \set{t} \in \mathcal{T}_{S,4}(d)$ if and only if $t \notin T$.
    Moreover, for $t \in S'$, the only element of $\mathcal{T}_{S,3}(t+d)$ that contains $t$ is $\set{t,s_1, s_2}$ since $S$ is Sidon.
    Hence, for any $t \in S'$, we have $\mult_{S,3}(t+d)=1+k_t^{(4)}$. 

    Now, let $\ell = \floor{\frac{|S|}{3}}=\frac{|S|-2}{3}$.
    By \Cref{prop:2-nontriple-pts}, for any point $a \in \maxmults{S}$ the points $s_1, s_2$ cannot be contained in a triple of three points in $S$ that sum to $a$.
    So, for any $t \in S'$, we have $t+d \notin \maxmults{S}$ as $\set{t,s_1, s_2} \in \mathcal{T}_{S,3}(t+d)$, implying $k_t^{(4)} = \mult_{S,3}(t+d) -1 \leq\ell-2$.
    Thus, 
    \[
    \mult_{S,4}(d) =
    \frac{1}{4}\sum_{t \in S'}k_t^{(4)}
    \leq \frac{(|S|-2)(\ell-2)}{4}=\frac{(|S|-2)(|S|-8)}{12}.
    \]
\end{proof}

We are now ready to prove the following theorem, which directly implies \Cref{thm:main}.

\begin{theorem}
    Assume $n \geq 6$.
    Let $S\subseteq \F_2^n$ be a Sidon set.
    Then $|S|$ is strictly less than the smallest integer $r_n$ such that $r_n >\sqrt{2^{n+1}-8}-2$ and $r_n \equiv 2 \mod 6$.
\end{theorem}
\begin{proof}
    Assume that $|S| \equiv 2\mod 6$ and $\maxmults{S}\neq \emptyset$ and $|S|>2$ (note that no such Sidon exists when $n\leq 5$, see \cite{quadspaper}).
    Let $s_1, s_2 \in S$ be defined as in \Cref{prop:2-nontriple-pts}, and let $d=s_1+s_2$.
    Recall that since $S$ is Sidon, the set $D = S\widehat{+}S$ has size $\binom{|S|}{2}$.
    By \Cref{lem:maxmults-upperbound-by-deltaA}, we have $|\maxmults{S}| \leq \Delta_A(d)$, and 
    \begin{align*}
    \Delta_A(d) &= 2^n-2\binom{|S|}{2} + \Delta_D(d)\\
    &= 2^n-2\binom{|S|}{2} + 2(|S|-2)+6\mult_{S,4}(d) \\
    &\leq 2^n-2\binom{|S|}{2} + 2(|S|-2)+ \frac{(|S|-2)(|S|-8)}{2} \\
    &= 2^n-\frac{|S|^2}{2}-2|S| +4
    \end{align*}
    with the first, second, and third (in)equalities following from \Cref{lem:deltaA-deltaD}, \Cref{lem:deltaD-multS4}, and \Cref{lem:multS4-bound} respectively.
    Combining the above with inequality (\ref{eq:lowerbound-on-maxmults-size2mod3}), we have
    \begin{align*}
       \frac{|S|^3 -|S|^2 -8|S| -2^{n+1}|S| + 5 \cdot 2^{n+1}}{6} 
       \leq |\maxmults{S}|  
       \leq 2^n-\frac{|S|^2}{2}-2|S| +4,
    \end{align*}
    which provides $|S|^3 + 2|S|^2 + 4|S| -2^{n+1} |S| + 2^{n+2} -24 \leq 0$, and this is the same as $(|S|-2)(|S|^2+4|S|+12-2^{n+1}) \leq 0$.
    Since $|S| > 2$, this inequality is equivalent to $|S|^2 + 4|S|+12 \leq 2^{n+1}$. 
    Upon completing the square, we have $(|S|+2)^2 +8 \leq 2^{n+1}$, i.e. $|S| \leq \sqrt{2^{n+1}-8}-2$.

    Now, let us no longer consider when $|S| \equiv 2 \mod 6$ nor $\maxmults{S} \neq \emptyset$, and assume instead that $|S| \geq r_n$.
    Let $S' \subseteq S$ such that $|S'| = r_n$.
    Applying inequality (\ref{eq:lowerbound-on-maxmults-size2mod3}), we have 
    $r_n^3 -r_n^2 -8r_n - 2^{n+1}r_n + 5 \cdot 2^{n+1} \leq 6 |\maxmults{S'}|$, and the left-hand side of this inequality is equal to $(r_n-5)(r_n^2+4r_n-2^{n+1}+12) + 60$.
    Since $n \geq 6$, we have $r_n > \sqrt{2^7-8}-2 > 8$, implying $r_n -5 >0$.
    Moreover, 
    \begin{align*}
        r_n^2 + 4r_n-2^{n+1}+12 &> (\sqrt{2^{n+1}-8}-2)^2 +4(\sqrt{2^{n+1}-8}-2) - 2^{n+1}+12=0.
    \end{align*}
    Hence, $|\maxmults{S'}|>0$, and the first part of the argument applies, providing a contradiction to the existence of $S$.
\end{proof}

\Cref{thm:main} immediately follows. 
Indeed, if $n \geq 7$ and $n \equiv 3 \mod 4$, then $r_n=2^{\frac{n+1}{2}} - 2$ as $2^{\frac{n+1}{2}}-2 \equiv 2 \mod 6$ and $2^{\frac{n+1}{2}}-2 > \sqrt{2^{n+1}-8}-2 > 2^{\frac{n+1}{2}} - 8$.
In particular, when $n \geq 15$ and $n \equiv 3 \mod 4$, this is a strict improvement on the best previously known upper bound on $s_{\max}(\F_2^n)$ \cite{CzerwinskiPottLargeSidon} (note that we do not obtain an improvement when $n \in \set{7,11}$ since it is already known that $s_{\max}(\F_2^7)=12<14=2^{\frac{7+1}{2}}-2$ and $s_{\max}(\F_2^{11})\leq58<62=2^{\frac{11+1}{2}}-2$).

\section{Connections to coding theory and APN functions}\label{sec:cor}

A (binary) \textit{linear code} $C$ is a $k$-dimensional linear subspace in $\F_2^n$.
We call $n$ the \textit{length} of $C$, and we call $k$ the \textit{dimension} of $C$.
The \textit{minimum distance} of $C$ is the integer $d$ satisfying $d = \min_{c \in C \setminus \set{0}} \wt(c)$, where $\wt(c)$ is the \textit{Hamming weight} of $c$ and is equal to the number of nonzero entries of the vector $c$. 
We then say that $C$ is a $[n,k,d]$-code.
In \cite{CzerwinskiPott2024}, the following connection between the existence of particular binary linear codes of minimum distance $5$ and Sidon sets was established.

\begin{corollary}{\cite[Corollary 5.1]{CzerwinskiPott2024}}\label{cor:code-Sidon}
    Let $n,s$ be integers satisfying $s > n \geq 2$.
    There does not exist a $[s,s-n,5]$-code if and only if $s_{\max}(\F_2^n) \leq s$.
\end{corollary}

In \cite{Brouwer1993}, Brouwer and Tolhuizen proved that there is no $[2^t-2, 2^t-2t-1,5]$-code for all $t >3$, or equivalently, there is no $[2^{\frac{n+1}{2}}-2, 2^{\frac{n+1}{2}} -n-2, 5]$-code for all odd $n \geq 7$.
Applying \Cref{cor:code-Sidon} to \Cref{thm:main}, we improve the Brouwer-Tolhuizen result for every even $t \geq 4$, and to the best of the our knowledge, this is the first improvement on Brouwer and Tolhuizen's result since its publication in 1993.
\begin{corollary}
    Assume $n \geq 7$ and $n \equiv 3 \mod 4$.
    Then, there is no $[2^{\frac{n+1}{2}}-3,2^{\frac{n+1}{2}}-n-3,5]$-code.
    Equivalently, for any even $t\geq4$, there is no $[2^t-3,2^t-2t-2,5]$-code.
\end{corollary}

For a function $F\colon \F_2^n \to \F_2^n$, we say that $F$ is \textit{differentially $\delta$-uniform} if for any nonzero $a \in \F_2^n$ and any $b\in \F_2^n$, the number of solutions to the equation $F(x)+F(x+a)=b$ is at most $\delta$.
If $F$ is differentially 2-uniform, then we say that $F$ is \textit{almost perfect nonlinear} (APN).
It is well-known that $F$ is APN if and only if its \textit{graph} 
$\graph{F}=\set{(x,F(x)) : x \in \F_2^n}$ is a Sidon set in $\F_2^n \times \F_2^n$ (cf. \cite{Carlet_apnGraphMaximal}).

A well-known problem on APN functions is to improve the lower bound on their nonlinearity.
For a Boolean function $f \colon \F_2^n \to \F_2$, we define the \textit{nonlinearity} $\mathcal{NL}(f)$ of $f$ as the minimum Hamming distance between $f$ and an affine function, or equivalently, 
$
\mathcal{NL}(f) = 2^{n-1} - \frac 12 \max_{u \in \F_2^n} |W_f(u)|,
$
where $W_f(u) = \sum_{x\in \F_2^n}(-1)^{f(x) + x \cdot u}$ is the \textit{Walsh transform} of $f$.
The \textit{nonlinearity} of a function $F \colon \F_2^n\to \F_2^n$ is defined to be $\mathcal{NL}(F) = \min_{v \in \F_2^n \setminus \set{0}}\mathcal{NL}(v \cdot F)$, where $(v \cdot F)(x) = v \cdot F(x)$.
The \textit{Walsh transform} of $F$ is defined by $W_F(u,v) =W_{v \cdot F}(u)= \sum_{x\in \F_2^n}(-1)^{u\cdot x+ v\cdot F(x)}$, so the nonlinearity of $F$ is equivalently defined by $\mathcal{NL}(F) = 2^{n-1}-\frac{1}{2}\max_{(u,v) \in \F_2^n \times \F_2^n, v \neq 0} |W_F(u,v)|$.
For any $(u,v) \in(\F_2^n \times \F_2^n)\setminus \set{(0,0)}$, it is clear that 
\begin{equation}\label{eq:WF-intersections}
W_F(u,v) = 2 |\graph{F} \cap \set{(0,0),(u,v)}^\perp| - 2^n=2^n -2|\graph{F} \setminus \set{(0,0),(u,v)}^\perp|.
\end{equation}
Assume that $n \geq 4$ is even and that $F$ is APN.
We then know that $\graph{F}$ cannot intersect a linear hyperplane or its complement in at most $2$ points as $s_{\max}(\F_2^{2n-1}) \leq 2^n-3$ by \Cref{thm:main}.
From relation~(\ref{eq:WF-intersections}), we then know that $|W_F(u,v)| \leq 2^n-6$ for all nonzero $(u,v)$.
This immediately yields the following corollary, strengthening \cite[Corollary 3.7]{CzerwinskiPottLargeSidon}, where it was shown that any APN function in $n \geq 3$ variables has nonlinearity at least $2$.

\begin{corollary}
    Assume $n \geq 4$ is even.
    For any APN function $F \colon \F_2^n \to \F_2^n$, we have $\mathcal{NL}(F) \geq 3$.
\end{corollary}

\section*{Acknowledgments}
The author thanks Claude Carlet for useful comments that helped to improve on earlier drafts of this paper.
AI assistance was used for manuscript revision, proofreading, grammatical corrections, and checking mathematical arguments.

\bibliographystyle{plain}
\bibliography{references}
\end{document}